\documentclass[a4paper,12pt]{article}

\usepackage[utf8]{inputenc}
\usepackage[english]{babel}
\usepackage{enumitem}
\usepackage{float}
\usepackage{amsmath}
\usepackage{amssymb}
\usepackage{amsthm}
\usepackage{amsfonts}
\usepackage{booktabs}
\usepackage{epsfig}
\usepackage{epstopdf}
\usepackage{cite}
\usepackage{algorithm,algpseudocode}

\usepackage{caption}
\newtheorem{dfn}{Definition}[section]
\newtheorem{lem}[dfn]{Lemma}
\newtheorem{thm}[dfn]{Theorem}

\theoremstyle{definition}

\newtheorem{asm}[dfn]{Assumption}
\newtheorem{exm}{Example}[section]
\newtheorem{rem}[dfn]{Remark}

\allowdisplaybreaks

\title{Convergence Analysis of Projection Method for Variational Inequalities}

\date{January 22, 2021}

\author{Yekini
Shehu\footnote{Department of Mathematics, Zhejiang Normal University, Jinhua,
321004, People's Republic of China; Institute of Science and Technology (IST), Am Campus 1, 3400, Klosterneuburg, Vienna, Austria; e-mail: yekini.shehu@unn.edu.ng.}
\hspace*{0.8mm}   Olaniyi. S. Iyiola\footnote{Department of Mathematics, Minnesota State University-Moorhead, Minnesota, USA; e-mail: olaniyi.iyiola@mnstate.edu}
\hspace*{0.8mm}
Xiao-Huan Li\footnote{College of Science, Civil Aviation University of China, Tianjin 300300, China.; e-mail: xiaohuanlimath@163.com.}\,\, and Qiao-Li Dong\footnote{College of Science, Civil Aviation University of China, Tianjin 300300, China.; e-mail: dongql@lsec.cc.ac.cn.}}

\begin{document}

\maketitle

\begin{abstract}
\noindent
The main contributions of this paper are the proposition and the convergence analysis of a class of
inertial projection-type algorithm for solving variational inequality problems in real Hilbert spaces where the underline operator is monotone and uniformly continuous. We carry out a unified analysis of the proposed method under very mild assumptions. In particular, weak convergence of the generated sequence is established and nonasymptotic $O(1/n)$ rate of convergence is established, where $n$ denotes the iteration counter. We also present some experimental results to illustrate the profits gained by introducing the inertial extrapolation steps.
\end{abstract}

\section{Introduction}\label{Sec:Intro}
\noindent
We first state the formal definition of some classes of functions that
play an essential role in this paper.\\

\noindent
Let $ H $ be a real Hilbert space and $ X \subseteq H $ be a nonempty subset.
\begin{dfn}\label{Def:LipMon}
A mapping $ F: X \to H $
is called
\begin{itemize}
\item[{\rm(a)}] {\em monotone} on $X$ if $\langle
F(x) - F(y), x-y\rangle\geq0$ for all $x,y\in X$;
\item[{\rm(b)}] {\em Lipschitz continuous} on $X$ if there exists a
      constant $L>0$ such that
      $$\|F(x) - F(y)\|\leq L\|x-y\|,\ \forall x,y\in X.$$
\item[{\rm(c)}] {\em sequentially weakly continuous} if for each sequence $\{x_n\}$ we have:  $\{x_n\}$ converges weakly to $x$ implies
$\{F(x_n)\}$ converges weakly to $F(x)$.

\end{itemize}
\end{dfn}

\noindent
Let $C$ be a nonempty, closed and convex subset of $H$ and $F:C\rightarrow H$ be a continuous mapping. The variational inequality problem (for short,
VI($F,C$)) is defined as: find $x\in C$ such that
\begin{eqnarray}\label{bami2}
   \langle F(x), y-x\rangle \geq 0, \quad \forall y \in C.
\end{eqnarray}
Let $ \text{SOL} $ denote the solution set of VI($F,C$) \eqref{bami2}. Variational inequality theory is an important tool in economics, engineering mechanics, mathematical programming, transportation, and so on (see, for example, \cite{Aubin, Baiocchi, Glowinski, Khobotov, Kinderlehrer, Konnov, Marcotte}). \\

\noindent
A well-known projection-type method for solving VI($F,C$) \eqref{bami2} is the extragradient method introduced by
Korpelevich in \cite{Korpelevich}. It is well known that the extragradient method requires two
projections onto the set $C$ and two evaluations of $F$ per iteration.\\

\noindent
One important hallmark in the design of numerical methods related to the
extragradient method is to minimize the number of evaluations of $P_C$ per
iteration because if $C$ is a general closed and convex set,
then a minimal distance problem has to be solved (twice) in order to obtain the next
iterate. This has the capacity to seriously affect the efficiency of the extragradient method in a situation,
where a projection onto $C$ is hard to evaluate and therefore
computationally costly.\\

\noindent
An attempt in this direction was initiated by
Censor et al.\cite{Censor2}, who modified extragradient method
by replacing the second projection onto the closed and convex
subset $ C $ with the one onto a subgradient half-space. Their method, which
therefore uses only one projection onto $ C $, is called
the subgradient extragradient method: $x_1\in H$,
\begin{eqnarray}\label{pppp3}
\left\{  \begin{array}{llll}
         & y_n = P_C(x_n-\lambda F(x_n)),\\
         & T_n:=\{w \in H:\langle x_n-\lambda F(x_n)-y_n,w-y_n\rangle \leq 0\},\\
         & x_{n+1}=P_{T_n}(x_n-\lambda F(y_n)).
         \end{array}
         \right.
\end{eqnarray}
Using \eqref{pppp3}, Censor {\it et al.} \cite{Censor2} proved weak convergence result for VI($F,C$) \eqref{bami2} with a monotone and $L$-Lipschitz-continuous mapping $ F $ where $\lambda \in (0,\frac{1}{L})$. Several other related methods to extragradient method and \eqref{pppp3} for solving VI($F,C$) \eqref{bami2} in real Hilbert spaces when $F$ is monotone and $L$-Lipschitz-continuous mapping have been studied in the literature (see, for example, \cite{Ceng, CengYao1, Censor1, Dong, Hieu, MaingePE1, Malitsky3, Malitsky, Nadezhkina, Tseng}).\\

\noindent
Motivated the result of Alvarez and Attouch in \cite{Alvarez} and Censor {\it et al.} in \cite{Censor2},
Thong and Hieu \cite{Thong} introduced an algorithm which is a combination of \eqref{pppp3} and inertial method for solving VI($F,C$) \eqref{bami2} in real Hilbert space: $x_0, x_1\in H$,
\begin{eqnarray}\label{mm2}
\left\{  \begin{array}{llll}
         & w_n=x_n+\alpha_n(x_n-x_{n-1}),\\
         & y_n = P_C(w_n-\lambda F(w_n)),\\
         & T_n:=\{w \in H:\langle w_n-\lambda F(w_n)-y_n,w-y_n\rangle \leq 0\},\\
         & x_{n+1}=P_{T_n}(w_n-\lambda F(y_n)).
         \end{array}
         \right.
\end{eqnarray}
Thong and Hieu \cite{Thong} proved that the sequence $\{x_n\}$ generated by \eqref{mm2} converges weakly
to a solution of VI($F,C$) \eqref{bami2} with a monotone and $L$-Lipschitz-continuous mapping $ F $ where $0<
\lambda L\leq \frac{\frac{1}{2}-2\alpha-\frac{1}{2}\alpha^2-\delta}{\frac{1}{2}-\alpha+\frac{1}{2}\alpha^2}$ for some
$0<\delta<\frac{1}{2}-2\alpha-\frac{1}{2}\alpha^2$ and $\{\alpha_n\}$ is a non-decreasing sequence with $0\leq \alpha_n \leq \alpha<\sqrt{5}-2$.\\

\noindent
One the main features of the above mentioned methods \eqref{pppp3}, \eqref{mm2} and other related methods is the computational issue, for example, step-sizes. The step-sizes in these above methods are bounded by the inverse of the Lipschitz constant which is quite inefficient, because in most cases a global Lipschitz constant (if it indeed exists) of $F$ cannot be accurately estimated, and is usually overestimated, thereby resulting in too small step-sizes. This, of course, is not practical.
Therefore, algorithms \eqref{pppp3} and \eqref{mm2} are not applicable in most cases of
interest. The usual approach to overcome this difficulty consists in some prediction
of a step-size with its further correction (see \cite{Khobotov, Marcotte}) or in a usage of an Armijo type
line search procedure along a feasible direction (see \cite{Solodov}). In terms of computations, the latter
approach is more effective, since very often the former approach requires too many
projections onto the feasible set per iteration.\\

\noindent
This paper focuses on the analysis and development of computational projection-type algorithm with inertial extrapolation step for solving VI($F,C$) \eqref{bami2} when the underline operator $F$ is monotone and uniformly continuous when the feasible set $C$ is a nonempty closed affine subset. We obtain weak convergence of the sequence generated by our method. We provide theoretical analysis of our result with weaker assumption on the underline operator $F$ unlike \cite{Censor1, Censor2, MaingePE1, Malitsky3} and many other related results on monotone variational inequalities. We also establish the nonasymptotic $O(1/n)$ rate of convergence, which is not given before in other previous inertial type projection methods for VI($F,C$) \eqref{bami2} (see, e.g.,\cite{Dong,Thong}) and give carefully designed computational experiments to illustrate our results. Our computational results show that our proposed methods outperform the iterative methods \eqref{pppp3} and \eqref{mm2}. Furthermore, our result  complements some recent results on inertial type algorithms (see, e.g., \cite{Alvarez, Attouch, Attouch2, Attouch3, AttouchPey, Beck, Bot2, Bot3, CChen, Lorenz, Mainge2, Ochs, Polyak2}).\\

\noindent
The paper is organized as follows: We first recall some basic
definitions and results in Section~\ref{Sec:Prelims}. Some discussions
about the proposed inertial projection-type method are given in
Section~\ref{Sec:Method}. The weak convergence analysis of
our algorithm is then investigated
in Section~\ref{Sec:Convergence}. We give the rate of convergence of our proposed method in Section~\ref{Sec:Rate} and some numerical experiments can be found in Section~\ref{Sec:Numerics}. We conclude
with some final remarks in Section~\ref{Sec:Final}.

\section{Preliminaries}\label{Sec:Prelims}

\noindent First, we recall some properties of the projection, cf.\
\cite{Bauschkebook} for more details. For any point $u \in H$, there exists a unique point $P_C u
\in C$ such that
$$\|u-P_C u\|\leq\|u-y\|,~\forall y \in C.$$
$P_C$ is called the {\it metric projection } of $H$ onto $C$. We
know that $P_C$ is a nonexpansive mapping of $H$ onto $C$. It is
also known that $P_C$ satisfies
\begin{equation}\label{a21}
   \langle x-y, P_C x-P_C y \rangle \geq \|P_C x-P_C y\|^2,~~\forall x, y \in H.
\end{equation}
In particular, we get from \eqref{a21} that
\begin{equation}\label{load}
   \langle x-y, x-P_C y \rangle \geq \|x-P_C y\|^2,~~\forall x \in C, y \in H.
\end{equation}
Furthermore, $P_C x$ is characterized by the properties
\begin{equation}\label{a22}
   P_Cx\in C\quad \text{and} \quad \langle x-P_C x,P_C x-y\rangle\geq0,~\forall y\in C.
\end{equation}
Further properties of the metric projection can be found, for example, in Section 3 of \cite{Goebel}.\\

\noindent
The following lemmas will be used in our convergence analysis.

\begin{lem}\label{lm2}
The following statements hold in $ H $:
\begin{itemize}
   \item[(a)] $ \|x+y\|^2=\|x\|^2+2\langle x,y\rangle+\|y\|^2 $ for all
      $ x, y \in H $;
   \item[(b)] $ 2 \langle x-y, x-z \rangle = \| x-y \|^2 + \| x-z \|^2 -
      \| y-z \|^2 $ for all $ x,y,z \in H $;
   \item[(c)] $\|\alpha x+(1-\alpha)y\|^2 =\alpha \|x\|^2+(1-\alpha)\|y\|^2
   -\alpha(1-\alpha)\|x-y\|^2$ for all $x,y \in H$ and $\alpha \in \mathbb{R}$.
%   \item[(d)] $\|x+y\|^2\leq\|x\|^2+2\langle y,x+y\rangle$ for all $x, y \in H$.

\end{itemize}
\end{lem}

\begin{lem}\label{la5}(see \cite[Lem.\ 3]{alvarez}) Let $\{ \psi_n\}$, $\{ \delta_n\}$ and $\{ \alpha_n\}$ be the sequences in $[0,+\infty)$ such that $\psi_{n+1}\leq \psi_n+\alpha_n(\psi_n-\psi_{n-1})+\delta_n$ for all $n\geq1$, $\sum_{n=1}^{\infty}\delta_n<+\infty$ and there exists a real number $\alpha$ with $0\leq\alpha_n\leq\alpha<1$ for all $n\geq1$. Then the following hold:\\
$(i)~~\sum_{n\geq1}[\psi_n-\psi_{n-1}]_+<+\infty$, where $[t]_+=\max\{ t,0\}$;\\
(ii) there exists $\psi^*\in[0,+\infty)$ such that $\lim_{n\rightarrow+\infty}\psi_n=\psi^*$.
\end{lem}

\begin{lem}\label{super}(see \cite[Lem.\ 2.39]{Bauschkebook})
Let $C$ be a nonempty set of $H$ and $\{ x_n\}$ be a sequence in $H$ such that the following two conditions hold:\\
(i) for any $x\in C$, $\lim_{n\rightarrow\infty}\| x_n-x\|$ exists;\\
(ii) every sequential weak cluster point of $\{ x_n\}$ is in $C$.\\
Then $\{ x_n\}$ converges weakly to a point in $C$.
\end{lem}

\noindent The following lemmas were given in $\mathbb{R}^n$ in
\cite{He2}. The proof of the lemmas are the same if given in
infinite dimensional real Hilbert spaces. Hence, we state the lemmas
and omit the proof in real Hilbert spaces.

\begin{lem}\label{new1}
Let $C$ be a nonempty closed and convex subset of $H$. Let $h$ be a real-valued function on $H$ and define
$K:=\{x: h(x)\leq 0\}$. If $K$ is nonempty and $h$ is
Lipschitz continuous on $C$ with modulus $\theta>0$, then
$${\rm dist}(x,K) \geq\theta^{-1}\max \{h(x),0\},~\forall x \in C,$$
\noindent where ${\rm dist}(x,K)$ denotes the distance function from
$x$ to $K$.
\end{lem}

\begin{lem}\label{new3}
Let $C$ be a nonempty closed and convex subset of $H$, $y:=P_C(x)$ and $x^* \in C$. Then
\begin{equation}\label{SpCseEst}
    \|y-x^*\|^2\leq\|x-x^*\|^2-\|x-y\|^2.
\end{equation}
\end{lem}

\noindent
The following lemma was stated in \cite[Prop.\ 2.11]{Iusem1}, see
also \cite[Prop.\ 4]{Iusem3}.

\begin{lem}\label{lm28a}
Let $H_1$ and $H_2$ be two real Hilbert spaces. Suppose
$F:H_1\rightarrow H_2$ is uniformly continuous on bounded subsets of $H_1$
and $M$ is a bounded subset of $H_1$. Then $F(M)$ is bounded.
\end{lem}

\noindent
Finally, the following result states the equivalence
between a primal and a weak form of variational inequality for continuous,
monotone operators.

\begin{lem}\label{lm27a}(\cite[Lem.\ 7.1.7]{wtak})
Let $C$ be a nonempty, closed, and convex subset of $H$.
Let $F:C\rightarrow H$ be a continuous, monotone mapping and $z \in C$. Then
$$
   z \in {\rm SOL} \Longleftrightarrow \langle F(x), x-z\rangle
   \geq 0 \quad \text{for all } x\in C.
$$
\end{lem}

\section{Inertial Projection-type Method}\label{Sec:Method}

Let us first state the assumptions
that we will assume to hold for the rest of this paper.

\begin{asm}\label{Ass:VI}
Suppose that the following hold:
\begin{itemize}
\item[{\rm(a)}] The feasible set $ C $ is a nonempty closed affine
      subset of the real Hilbert space $ H $.
\item[{\rm(b)}] $ F: C \to H $ is monotone and  uniformly continuous on bounded subsets of $H$.
\item[{\rm(c)}] The solution set $\text{SOL}$ of VI$(F,C)$ is nonempty.
\end{itemize}
\end{asm}

\begin{asm}\label{Ass:Parameters}
Suppose the real sequence $ \{ \alpha_n \} $ and constants $ \beta, \delta, \sigma>0$
satisfy the following conditions:
\begin{itemize}
   \item[(a)] $ \{ \alpha_n \} \subset (0,1) $
       with  $0\leq \alpha_n \leq \alpha_{n+1} \leq \alpha<1$ for all $n$.
   \item[(b)] $\delta> \frac{\alpha(1+\alpha)(\alpha+\delta \sigma)+\alpha\sigma\delta(\alpha+\delta\sigma)}{\sigma}$
   and $\beta<\frac{\delta\sigma}{\alpha+\delta\sigma}-\alpha(1+\alpha)-\alpha\sigma\delta$.
\end{itemize}
\end{asm}

\noindent
Let
$$
   r(x) := x - P_C (x-F(x))
$$
stand for the residual equation.

\noindent Observe that if we take $y=x-F(x)$ in \eqref{load}, then we
have
\begin{equation}\label{load2}
   \langle F(x),r(x)\rangle\geq\|r(x)\|^2,~\forall x\in C.
\end{equation}
\noindent We now introduce our proposed method below.

\begin{algorithm}[H]
\caption{Inertial Projection-type Method}\label{Alg:AlgL}
\begin{algorithmic}[1]
\State Choose sequence $ \{ \alpha_n \} $ and $\sigma \in (0,1)$ such that the conditions from Assumption~\ref{Ass:Parameters} hold,
      and take $\gamma \in (0,1)$.
      Let $ x_0= x_1 \in H $ be a given starting point. Set $ n := 1 $.
\State Set
       \begin{equation*}
         w_n := x_n+\alpha_n (x_n-x_{n-1}).
       \end{equation*}
Compute $z_n:=P_C(w_n-F(w_n))$. If $r(w_n)=w_n-z_n=0$: STOP.
\State Compute $y_n=w_n-\gamma^{m_n}r(w_n)$, where $m_n$ is the smallest nonnegative integer satisfying
         \begin{equation}\label{ee31}
         \langle F(y_n),r(w_n)\rangle\geq\frac{\sigma}{2}\|r(w_n)\|^2.
         \end{equation}
Set $\eta_n:=\gamma^{m_n}$.
\State Compute
         \begin{equation}\label{e31}
         x_{n+1}=P_{C_n}(w_n),
         \end{equation}
where $C_n=\{x: h_n(x)\leq0\}$ and
         \begin{equation}\label{ego}
         h_n(x):=\langle F(y_n),x-y_n\rangle.
         \end{equation}
\State Set $n\leftarrow n+1$ and \textbf{goto 2}.
\end{algorithmic}
\end{algorithm}
\noindent
It is clear that $ r(w_n) = 0 $ implies that we are at a solution of the
variational inequality. In our convergence theory, we will implicitly
assume that this does not occur after finitely many iterations, so that
Algorithm~\ref{Alg:AlgL} generates an infinite sequence satisfying,
in particular, $ r(w_n) \neq 0 $ for all $ n \in \mathbb{N} $. We will
see that this property implies that Algorithm~\ref{Alg:AlgL} is
well defined.

\begin{rem}\label{Rem:Simple}
(a) Algorithm~\ref{Alg:AlgL} requires, at each iteration, only one
projection onto the feasible set $C $ and another projection onto the half-space $C_n$ (see \cite{Cegielskibook} for formula for computing projection onto half-space), which is less expensive than the extragradient method
especially for the case when computing the projection onto the feasible set $C$ is a dominating task during iteration. \\[-1mm]
%In addition,
%(a) Observe that in Algorithm~\ref{Alg:AlgL}, $x_n \in C, \forall n \geq 1$ and hence $x_n=P_C(x_n), \forall n \geq 1$. If $\alpha_n=0$, then computing projection in Step 2 is not needed. Thus, $w_n=x_n$ in this case. In addition, if $C$ is nonempty closed affine subset of $H$, then Step 2 of Algorithm~\ref{Alg:AlgL} does not need the projection onto $C$ and this reduces to $w_n=x_n+\alpha_n (x_n-x_{n-1})$.\\[-1mm]

\noindent (b) Our Algorithm~\ref{Alg:AlgL} is much more applicable than \eqref{pppp3} and \eqref{mm2} in the sense that algorithm \eqref{pppp3} and \eqref{mm2} are applicable only for \textit{monotone and $L$-Lipschitz-continuous mapping} $F$. Thus, the $L$-Lipschitz constant of $F$ or an estimate of it is needed in order to implement the iterative method \eqref{pppp3} but our Algorithm~\ref{Alg:AlgL} is applicable for a much more general class of \textit{monotone and uniformly continuous mapping} $F$. \\[-1mm]

\noindent (c) We observe that the step-size rule in Step 3 involves a couple
of evaluations of $ F $, but these are often much less expensive
than projections onto $ C $ which was considered in \cite{Khobotov, Marcotte}. Furthermore, using the fact that $F$ is continuous and \eqref{load2}, we can see that Step 3 in Algorithm~\ref{Alg:AlgL} is
well-defined.\hfill $\Diamond$
\end{rem}

\begin{lem}\label{new2}
Let the function $h_n$ be defined by
\eqref{ego}. Then
$$h_n(w_n)\geq\frac{\sigma\eta_n}{2}\|w_n-z_n\|^2.$$
\noindent In particular, if $w_n\neq z_n$, then
$h_n(w_n)>0$. If $x^*\in \text{SOL}$, then $h_n(x^*)\leq 0$.
 \end{lem}

\begin{proof}
Since $y_n=w_n-\eta_n(w_n-z_n)$, using \eqref{ee31} we have
$$\begin{aligned}
    h_n(w_n)
    &=\langle F(y_n),w_n-y_n\rangle\\
    &=\eta_n\langle F(y_n),w_n-z_n\rangle\geq\eta_n\frac{\sigma}{2}\|w_n-z_n\|^2\geq0.
\end{aligned}$$
If $w_n\neq z_n$, then $h_n(w_n)>0$.
Furthermore, suppose $x^* \in \text{SOL}$. Then by Lemma \ref{lm27a} we have
$\langle F(x), x-x^*\rangle   \geq 0 \quad \text{for all } x\in C.$
In particular, $\langle F(y_n), y_n-x^*\rangle   \geq 0$ and hence $h_n(x^*) \leq 0.$
\end{proof}

\noindent
Observe that, in finding $\eta_n$, the operator $F$ is evaluated (possibly)
many times, but no extra projections onto the set $C$ are needed. This
is in contrast to a couple of related algorithms for the solution of
monotone variational inequalities where the calculation of a suitable
step-size requires (possibly) many projections onto $ C $, see, e.g.,
\cite{Denisov, Khobotov, Tseng}.

\section{Convergence Analysis}\label{Sec:Convergence}
\noindent
We present our main result in this section. To this end, we begin with a result
that shows that the sequence $ \{ x_n \} $ generated by
Algorithm~\ref{Alg:AlgL} is bounded under the given
assumptions.

\begin{lem}\label{nece1}
Let $\{x_n\}$ be generated by Algorithm~\ref{Alg:AlgL}. Then under Assumptions~\ref{Ass:VI}
and \ref{Ass:Parameters}, we have that $\{x_n\}$ is bounded.
\end{lem}

\begin{proof}
Let $x^*\in\text{SOL}$.
By Lemma~\ref{new3} we get (since $x^*\in
C_n$) that
\begin{eqnarray}\label{jj7}
    \|x_{n+1}-x^*\|^2&=&\|P_{C_n}(w_n)-x^*\|^2\leq\|w_n-x^*\|^2-\|x_{n+1}-w_n\|^2 \\
    &=&\|w_n-x^*\|^2-{\rm dist}^2(w_n,C_n)\nonumber.
\end{eqnarray}
Now, using Lemma \ref{lm2} (c), we have
%(recall that $x^*=P_C(x^*)$)
\begin{eqnarray}\label{asu4}
    \|w_n-x^*\|^2&=&\|(1+\alpha_n)(x_n-x^*)-\alpha_n(x_{n-1}-x^*)\|^2\nonumber\\
    &=& (1+\alpha_n)\|x_n-x^*\|^2-\alpha_n\|x_{n-1}-x^*\|\nonumber\\
    &&+\alpha_n(1+\alpha_n)\|x_n-x_{n-1}\|^2.
    \end{eqnarray}
Substituting \eqref{asu4} into \eqref{jj7}, we have
\begin{eqnarray}\label{asu5}
    \|x_{n+1}-x^*\|^2&\leq &(1+\alpha_n)\|x_n-x^*\|^2-\alpha_n\|x_{n-1}-x^*\|^2+\alpha_n(1+\alpha_n)\|x_n-x_{n-1}\|^2\nonumber \\
    &&-\|x_{n+1}-w_n\|^2.
\end{eqnarray}
We also have (using Lemma \ref{lm2} (a))
% and the fact that $x_{n+1}=P_C(x_{n+1})$)
\begin{eqnarray}\label{asu6}
    \|x_{n+1}-w_n\|^2&=&\|(x_{x_n+1}-x_n)-\alpha_n(x_n-x_{n-1})\|^2\nonumber \\
    &=& \|x_{x_n+1}-x_n\|^2+\alpha^2\|x_n-x_{n-1}\|^2\nonumber \\
    &&-2\alpha_n\langle x_{x_n+1}-x_n, x_n-x_{n-1}\rangle \nonumber \\
    &\geq& \|x_{x_n+1}-x_n\|^2+\alpha^2\|x_n-x_{n-1}\|^2\nonumber \\
    &&+\alpha_n\Big(-\rho_n\|x_{x_n+1}-x_n\|^2-\frac{1}{\rho_n} \|x_n-x_{n-1}\|^2\Big),
\end{eqnarray}
where
$\rho_n:=\frac{1}{\alpha_n+\delta\sigma}$.
Combining \eqref{asu5} and \eqref{asu6}, we get
\begin{eqnarray}\label{asu7}
 &&\|x_{n+1}-x^*\|^2-(1+\alpha_n)\|x_n-x^*\|^2+\alpha_n\|x_{n-1}-x^*\|^2\nonumber \\
&\leq& (\alpha_n\rho_n-1)\|x_{n+1}-x_n\|^2+\lambda_n\|x_n-x_{n-1}\|^2,
\end{eqnarray}
where
\begin{eqnarray}\label{asu8}
\lambda_n:=\alpha_n(1+\alpha_n)+\alpha_n \frac{1-\alpha_n\rho_n}{\rho_n}\geq 0
\end{eqnarray}
\noindent since $\alpha_n\rho_n <1$. Taking into account the choice of
$\rho_n$, we have
$$
\delta=\frac{1-\alpha_n\rho_n}{\sigma \rho_n}
$$
\noindent and from \eqref{asu8}, it follows that
\begin{eqnarray}\label{asu9}
\lambda_n&=&\alpha_n(1+\alpha_n)+\alpha_n \frac{1-\alpha_n\rho_n}{\rho_n}\nonumber \\
&\leq& \alpha(1+\alpha)+\alpha\sigma\delta.
\end{eqnarray}

Following the same arguments as in \cite{alvarez, Alvarez, Bot}, we define
$\varphi_n:=\|x_n-x^*\|^2, n \geq 1$ and
$\varepsilon_n:=\varphi_n-\alpha_n\varphi_{n-1}+\lambda_n\|x_n-x_{n-1}\|^2, n \geq 1.$
By the monotonicity of $\{\alpha_n\}$ and the fact that $\varphi_n \geq 0$, we have
$$
\varepsilon_{n+1}-\varepsilon_n \leq \varphi_{n+1}-(1+\alpha_n)\varphi_n
+\alpha_n\varphi_{n-1}+\lambda_{n+1}\|x_{n+1}-x_n\|^2-\lambda_n\|x_n-x_{n-1}\|^2.
$$
\noindent
Using \eqref{asu7}, we have
\begin{eqnarray}\label{asu10}
\varepsilon_{n+1}-\varepsilon_n&\leq& (\alpha_n\rho_n-1)\|x_{n+1}-x_n\|^2+\lambda_n\|x_n-x_{n-1}\|^2\nonumber\\
&&+\lambda_{n+1}\|x_{n+1}-x_n\|^2-\lambda_n\|x_n-x_{n-1}\|^2\nonumber \\
&=&(\alpha_n\rho_n-1+\lambda_{n+1})\|x_{n+1}-x_n\|^2.
\end{eqnarray}
We now claim that
\begin{eqnarray}\label{asu11}
\alpha_n\rho_n-1+\lambda_{n+1} \leq -\beta.
\end{eqnarray}
Indeed by the choice of $\rho_n$, we have
\begin{eqnarray*}
&&\alpha_n\rho_n-1+\lambda_{n+1} \leq -\beta\\
& \Leftrightarrow & \\
&&\alpha_n\rho_n-1+\lambda_{n+1} +\beta\leq 0\\
& \Leftrightarrow & \\
&& \lambda_{n+1} +\beta+\frac{\alpha_n}{\alpha_n+\delta\sigma}-1\leq 0\\
& \Leftrightarrow & \\
&&\lambda_{n+1} +\beta-\frac{\delta\sigma}{\alpha_n+\delta\sigma}\leq 0\\
& \Leftrightarrow & \\
&&(\alpha_n+\delta\sigma)(\lambda_{n+1} +\beta)\leq \delta\sigma
\end{eqnarray*}
\noindent Now, using \eqref{asu9}, we have
$$
(\alpha_n+\delta\sigma)(\lambda_{n+1} +\beta)
\leq ((\alpha+\delta\sigma)(\alpha(1+\alpha)\alpha\delta\sigma+\beta) \leq\delta\sigma,
$$
\noindent where the last inequality follows from Assumption~\ref{Ass:Parameters} (b). Hence, the claim
in \eqref{asu11} is true.\\

\noindent
Thus, it follows from \eqref{asu10} and \eqref{asu11} that
\begin{eqnarray}\label{asu12}
\varepsilon_{n+1}-\varepsilon_n\leq -\beta\|x_{n+1}-x_n\|^2.
\end{eqnarray}
The sequence $\{\varepsilon_n\}$ is non-increasing and the bounds of $\{\alpha_n\}$ delivers
\begin{eqnarray}\label{asu13}
-\alpha\varphi_{n-1}\leq\varphi_n -\alpha\varphi_{n-1} \leq \varepsilon_n \leq \varepsilon_1, n \geq 1.
\end{eqnarray}
It then follows that
\begin{eqnarray}\label{asu14}
\varphi_n\leq\alpha^n\varphi_0+\varepsilon_1\sum_{k=0}^{n-1}\alpha^k \leq \alpha^n\varphi_0+\frac{\varepsilon_1}{1-\alpha} , n \geq 1.
\end{eqnarray}
Combining \eqref{asu12} and \eqref{asu13}, we get
\begin{eqnarray}\label{addi}
\beta \sum_{k=1}^{n}\|x_{k+1}-x_k\|^2 &\leq& \varepsilon_1-\varepsilon_{n+1}\nonumber\\
&\leq& \varepsilon_1+\alpha \varphi_n\nonumber\\
&\leq& \alpha^{n+1}\varphi_0+\frac{\varepsilon_1}{1-\alpha}\nonumber\\
&\leq& \varphi_0+\frac{\varepsilon_1}{1-\alpha},
\end{eqnarray}
which shows that
\begin{eqnarray}\label{asu15}
\sum_{k=1}^{\infty}\|x_{k+1}-x_k\|^2<\infty.
\end{eqnarray}
Thus,
$\underset{n\rightarrow \infty}\lim \|x_{n+1}-x_n\|=0$.
From $w_n=x_n+\alpha_n(x_n-x_{n-1})$, we have
\begin{eqnarray*}
\|w_n-x_n\|&\leq&\alpha_n\|x_n-x_{n-1}\|\\
&\leq& \alpha \|x_n-x_{n-1}\|\rightarrow 0, n\rightarrow \infty.
\end{eqnarray*}
Similarly,
$$
\|x_{n+1}-w_n\| \leq \|x_{n+1}-x_n\|+\|x_n-w_n\|\rightarrow 0, n\rightarrow \infty.
$$
\noindent
Using Lemma \ref{la5}, \eqref{asu7}, \eqref{asu9} and \eqref{asu15}, we have that
$\underset{n\rightarrow \infty}\lim \|x_n-x^*\|$ exists. Hence, $\{x_n\}$ is bounded.
\end{proof}

\noindent In the next two lemmas, we show that certain subsequences
obtained in Algorithm~\ref{Alg:AlgL} are null subsequences. These
two lemmas are necessary in order to show that the weak limit of
$\{x_n\}$ is an element of $SOL$ and for our
weak convergence in Theorem~\ref{t32} below.

\begin{lem}\label{nece3}
Let $\{x_n\}$ generated by Algorithm~\ref{Alg:AlgL} above and
Assumptions~\ref{Ass:VI} and \ref{Ass:Parameters} hold.
Then
\begin{itemize}
\item[{\rm(a)}] $\displaystyle\lim_{n\rightarrow \infty} \eta_{n}\|w_n-z_n\|^2=0$;
\item[{\rm(b)}] $\displaystyle\lim_{n\rightarrow \infty} \|w_n-z_n\|=0.$
\end{itemize}
\end{lem}

\begin{proof}
Let $x^* \in \text{SOL} $.
Since $F$ is uniformly
continuous on bounded subsets of $X$, then $\{F(x_n)\}, \{z_n\},
\{w_n\}$ and $\{F(y_n)\}$ are bounded. In particular, there exists
$M>0$ such that $\|F(y_n)\|\leq M$ for all $n\in\mathbb{N}$. Combining
Lemma~\ref{new1} and Lemma~\ref{new2}, we get
\begin{eqnarray}\label{jjj7}
    \|x_{n+1}-x^*\|^2&=&\|P_{C_n}(w_n)-x^*\|^2\leq\|w_n-x^*\|^2-\|x_{n+1}-w_n\|^2 \nonumber \\
    &=&\|w_n-x^*\|^2-{\rm dist}^2(w_n,C_n)\nonumber\\
&\leq& \|w_n-x^*\|^2-\Big(\frac{1}{M}h_n(w_n)\Big)^2 \nonumber\\
    &\leq& \|w_n-x^*\|^2-\Big(\frac{1}{2M}\sigma\eta_n\|r(w_n)\|^2\Big)^2\nonumber\\
    &=&\|w_n-x^*\|^2-\Big(\frac{1}{2M}\sigma\eta_n\|w_n-z_n\|^2\Big)^2.
\end{eqnarray}
Since $\{x_n\}$ is bounded, we obtain from \eqref{jjj7} that
\begin{eqnarray}\label{asu13}
\Big(\frac{1}{2M}\sigma\eta_n\|w_n-z_n\|^2\Big)^2 &\leq& \|w_n-x^*\|^2- \|x_{n+1}-x^*\|^2\nonumber \\
&=&\Big(\|w_n-x^*\|- \|x_{n+1}-x^*\|\Big)\Big(\|w_n-x^*\|+ \|x_{n+1}-x^*\|\Big)\nonumber \\
&\leq& \|w_n-x^*\|- \|x_{n+1}-x^*\|M_1 \nonumber \\
&\leq& \|w_n-x_{n+1}\|M_1,
\end{eqnarray}
where $M_1:=\sup_{n\geq 1}\{\|w_n-x^*\|+ \|x_{n+1}-x^*\|\}$. This establishes (a).\\

\noindent To establish (b), We distinguish two cases depending on the behaviour of (the bounded)
sequence of step-sizes $\{\eta_n\}$.

\noindent \textbf{Case 1}: Suppose that $ \liminf_{n \to \infty}
\eta_n > 0 $. Then
$$0\leq \|r(w_n)\|^2=\frac{\eta_{n}\|r(w_n)\|^2}{\eta_n}$$
\noindent and this implies that
$$\begin{aligned}
    \limsup_{n\to\infty}\|r(w_n)\|^2
    &\leq\limsup_{n\to\infty}\bigg(\eta_n\|r(w_n)\|^2\bigg)
                             \bigg(\limsup_{n\to\infty} \frac{1}{\eta_n}\bigg)\\
    &=\bigg(\limsup_{n\to\infty}\eta_n\|r(w_n)\|^2\bigg)\frac{1}{\liminf_{n\to\infty}\eta_n}\\
    &=0.
\end{aligned}$$
Hence, $\limsup_{n \to \infty}\|r(w_n)\|=0$. Therefore,
$$\lim_{n\rightarrow\infty}\|w_n-z_n\|=\lim_{n\rightarrow\infty}\|r(w_n)\|=0.$$

\noindent \textbf{Case 2}: Suppose that $ \liminf_{n \to \infty}
\eta_n= 0 $. Subsequencing if necessary, we may assume without loss of generality that
$\lim_{n\to\infty}\eta_n=0$ and $\lim_{n\rightarrow \infty}\|w_n-z_n\|=a \geq 0$.

\noindent Define
$\bar{y}_n:=\frac{1}{\gamma}\eta_nz_n+\Big(1-\frac{1}{\gamma}\eta_n\Big)w_n$
or, equivalently,
$\bar{y}_n-w_n=\frac{1}{\gamma}\eta_n(z_n-w_n)$.
Since $\{z_n-w_n\}$ is bounded and since $\lim_{n \to
\infty}\eta_n=0$ holds, it follows that
\begin{equation}\label{ify2}
   \lim_{n\to\infty}\|\bar{y}_n-w_n\|=0.
\end{equation}
From the step-size rule and the definition of $\bar{y}_k$, we have
$$\langle F(\bar{y}_n),w_n-z_n\rangle<\frac{\sigma}{2}\|w_n-z_n\|^2,\ \forall n\in\mathbb{N},$$
or equivalently
$$
2\langle F(w_n),w_n-z_n\rangle+
2\langle F(\bar{y}_n)-F(w_n),w_n-z_n\rangle
<\sigma\|w_n-z_n\|^2,\ \forall n\in\mathbb{N}.$$
Setting $t_n:=w_n-F(w_n)$, we obtain form the last inequality that
$$
2\langle w_n-t_n,w_n-z_n\rangle
+
2\langle F(\bar{y}_n)-F(w_n),w_n-z_n\rangle
<\sigma\|w_n-z_n\|^2,\ \forall n\in\mathbb{N}.
$$
Using Lemma~\ref{lm2} (b)  we get
$$
2\langle w_n-t_n,w_n-z_n\rangle=\|w_n-z_n\|^2+\|w_n-t_n\|^2-\|z_n-t_n\|^2.
$$
Therefore,
$$
\|w_n-t_n\|^2-\|z_n-t_n\|^2 < (\sigma-1)\|w_n-z_n\|^2-2\langle F(\bar{y}_n)-F(w_n),w_n-z_n\rangle\ \forall n\in\mathbb{N}.
$$
Since $F$ is uniformly continuous on bounded subsets of $H$ and
\eqref{ify2}, if $a>0$ then the right hand side of the last inequality converges to $(\sigma-1)a<0$ as $n \to \infty$.
From the last inequality we have
$$
\limsup_{n \to \infty} \left( \|w_n-t_n\|^2-\|z_n-t_n\|^2 \right) \leq (\sigma-1)a < 0.
$$
For $\epsilon=-(\sigma-1)a/2 >0$, there
exists $N\in\mathbb{N}$ such that
$$\|w_n-t_n\|^2 - \|z_n-t_n\|^2 \leq (\sigma-1)a+\epsilon= (\sigma-1)a/2 <0  \quad \forall n\in \mathbb{N},n\geq N,$$
leading to
$$\|w_n-t_n\| < \|z_n-t_n\| \quad \forall n\in \mathbb{N},n\geq N, $$
which is a contradiction to the definition of
$z_n=P_C(w_n-F(w_n))$. Hence $a=0$,
which completes the proof.
\end{proof}

\noindent
The boundedness of the sequence $ \{ x_n \} $ implies that there
is at least one weak limit point. We show that such weak limit point belongs to
$SOL$ in the next result.

\begin{lem}\label{Lem:help}
Let Assumptions~\ref{Ass:VI} and \ref{Ass:Parameters} hold.
Furthermore let $\{x_{n_k}\}$ be a subsequence of $ \{x_n\}$
converging weakly to a limit point $ p $. Then $ p \in \text{SOL} $.
\end{lem}

\begin{proof}
By the definition of $z_{n_k}$ together with \eqref{a22}, we have
$$\langle w_{n_k}-F(w_{n_k})-z_{n_k},x-z_{n_k}\rangle\leq 0,\ \forall x\in C,$$
\noindent which implies that
$$\langle w_{n_k}-z_{n_k},x-z_{n_k}\rangle\leq\langle F(w_{n_k}),x-z_{n_k}
  \rangle,\ \forall x \in C.$$
Hence,
\begin{equation}\label{j11}
   \langle w_{n_k}-z_{n_k},x-z_{n_k}\rangle +\langle F(w_{n_k}),z_{n_k}-w_{n_k}\rangle
   \leq \langle F(w_{n_k}),x-w_{n_k}\rangle,\ \forall x\in C.
\end{equation}
Fix $x \in C$ and let $k\rightarrow \infty$ in \eqref{j11}.
Since
$\lim_{k \to \infty} \|w_{n_k}-z_{n_k}\|=0 $, we have
\begin{equation}\label{anbi}
    0\leq \liminf_{k \to \infty} \langle F(w_{n_k}),x-w_{n_k}\rangle
\end{equation}
for all $ x \in C $.
It follows from \eqref{j11} and the monotonicity of $F$ that
 \begin{eqnarray*}
 \langle w_{n_k}-z_{n_k},x-z_{n_k}\rangle +\langle F(w_{n_k}),z_{n_k}-w_{n_k}\rangle
    &\leq& \langle F(w_{n_k}),x-w_{n_k}\rangle\\
    &\leq& \langle F(x),x-w_{n_k}\rangle  \quad \forall x\in C.
 \end{eqnarray*}
Letting $k \to +\infty$ in the last inequality, remembering that  $\lim_{k \to \infty}\|w_{n_k}-z_{n_k}\|=0$  for all $k$, we have
$$
\langle F(x),x-p\rangle \geq 0 \quad  \forall x\in C.
$$
\noindent
In view of Lemma~\ref{lm27a}, this implies $p\in\text{SOL}$.
\end{proof}

\noindent All is now set to give the weak convergence result in the theorem below.

\begin{thm}\label{t32}
Let Assumptions~\ref{Ass:VI} and \ref{Ass:Parameters} hold.
Then the sequence $ \{x_n\} $ generated by Algorithm~\ref{Alg:AlgL}
weakly converges to a point in $\text{SOL} $.
\end{thm}

\begin{proof}
We have shown that \\
(i) $\lim_{n \to \infty} \|x_n-x^*\|$ exists;\\
(ii) $\omega_w(x_n) \subset \text{SOL}$, where
$\omega_w(x_n):=\{x:\exists x_{n_j}\rightharpoonup x\}$ denotes the weak $\omega$-limit set of $\{x_n\}$.\\
Then, by Lemma \ref{super}, we have that $\{x_n\}$ converges weakly to a point in $\text{SOL} $.
\end{proof}

\noindent
We give some discussions on further contributions of this paper in the remark below.

\begin{rem}
\noindent
(a) Our iterative Algorithm~~\ref{Alg:AlgL} is more applicable than some recent results on projection type methods with inertial extrapolation step for solving VI($F,C$) \eqref{bami2} in real Hilbert spaces. For instance, the proposed method in \cite{Dong} can only be applied for a case when $F$ is monotone and $L$-Lipschitz continuous. Moreover, the Lipschitz constant or an estimate of it has to be known when implementing the Algorithm 3.1 of \cite{Dong}. In this result, Algorithm~~\ref{Alg:AlgL} is applicable when $F$ is uniformly continuous and monotone operator.\\[-1mm]

(b) In finite-dimensional spaces, the assumption that $ F $ is
uniformly continuous on bounded subsets of $C$ automatically holds when $F$ is continuous.
Moreover, in this case, only continuity of $ F $ is required and our weak convergence in Theorem \ref{t32}
coincides with global convergence of sequence of iterates $\{x_n\}$ in $\mathbb{R}^n$. \\[-1mm]

(c) Lemmas 3.5, 4.1, 4.2 and Theorem 4.4 still hold for a more general case of $F$ pseudo-monotone (i.e., for all $x,y \in H$,
$\langle F(x),y-x\rangle\geq0\Longrightarrow\langle F(y),y-x\rangle\geq0;$).
We give a version of Lemma 4.3 for the case of $F$ pseudo-monotone in the Appendix.
\hfill $\Diamond$
\end{rem}

\section{Rate of Convergence}\label{Sec:Rate}
\noindent In this section we give the rate of convergence of the iterative method \ref{Alg:AlgL} proposed in Section \ref{Sec:Method}. We show that the proposed method has sublinear rate of convergence and establish the nonasymptotic
$O(1/n)$ convergence rate of the proposed method. To the best of our knowledge, there is no convergence rate result known
in the literature without stronger assumptions for inertial projection-type Algorithm \ref{Alg:AlgL} for VI$(F,C)$ \eqref{bami2} in infinite dimensional Hilbert spaces.

\begin{thm}\label{t33}
Let Assumptions~\ref{Ass:VI} and \ref{Ass:Parameters} hold.
Let the sequence $ \{x_n\} $ be generated by Algorithm~\ref{Alg:AlgL} and $x_0=x_1$. Then for any $x^* \in \text{SOL} $
and for any positive integer $n$, it holds that
$$
\underset{1 \leq i \leq n}\min \|x_{i+1}-w_i\|^2 \leq \frac{\Big[1+\Big(\frac{\alpha}{1-\alpha}+\frac{\alpha(1-\alpha)}{\beta} \Big)\Big(1+ \frac{1-\alpha_0}{1-\alpha}\Big) \Big]\|x_0-x^*\|^2}{n}.
$$
\end{thm}

\begin{proof}
From \eqref{asu5}, we have
\begin{eqnarray}\label{asusu1}
&&\|x_{n+1}-x^*\|^2-\|x_n-x^*\|^2-\alpha_n(\|x_n-x^*\|^2-\|x_{n-1}-x^*\|^2)\nonumber \\
&\leq& \alpha_n(1+\alpha_n)\|x_n-x_{n-1}\|^2-\|x_{n+1}-w_n\|^2.
\end{eqnarray}
This implies that
\begin{eqnarray}\label{asusu2}
\|x_{n+1}-w_n\|^2&\leq& \varphi_n-\varphi_{n+1}+\alpha_n(\varphi_n-\varphi_{n-1})+\delta_n\nonumber \\
&\leq& \varphi_n-\varphi_{n+1}+\alpha[V_n]_{+}+\delta_n,
\end{eqnarray}
where $\delta_n:=\alpha_n(1+\alpha_n)\|x_n-x_{n-1}\|^2$,
$V_n:=\varphi_n-\varphi_{n-1}$, $[V_n]_{+}:=\max \{V_n,0\}$ and
$\varphi_n:=\|x_n-x^*\|^2$.\\

\noindent Observe from \eqref{addi} that
$$
\sum_{n=1}^{\infty}\|x_{n+1}-x_n\|^2 \leq \frac{1}{\beta}\Big[\varphi_0+\frac{\varepsilon_1}{1-\alpha}\Big].
$$
So,
\begin{eqnarray}\label{asusu3}
\sum_{n=1}^{\infty}\delta_n&=& \sum_{n=1}^{\infty} \alpha_n(1+\alpha_n)\|x_n-x_{n-1}\|^2\nonumber\\
&\leq& \sum_{n=1}^{\infty} \alpha(1+\alpha)\|x_n-x_{n-1}\|^2\nonumber\\
&=& \alpha(1+\alpha)\sum_{n=1}^{\infty} \|x_n-x_{n-1}\|^2\nonumber\\
&\leq& \frac{\alpha(1+\alpha)}{\beta}\Big[\varphi_0+\frac{\varepsilon_1}{1-\alpha}\Big]:=C_1.
\end{eqnarray}
The inequality \eqref{asusu1} implies that
\begin{eqnarray*}
V_{n+1}&\leq &\alpha_n V_n+\delta_n\\
&\leq& \alpha[V_n]_{+}+\delta_n.
\end{eqnarray*}
Therefore,
\begin{eqnarray}\label{asusu4}
[V_{n+1}]_{+} &\leq& \alpha [V_n]_{+} +\delta_n\nonumber \\
&\leq& \alpha^n [V_1]_{+}+ \sum_{j=1}^{n}\alpha^{j-1} \delta_{n+1-j}.
\end{eqnarray}
Note that by our assumption $x_0=x_1$. This implies that $V_1=[V_1]_{+}=0$ and $\delta_1=0$.
From \eqref{asusu4}, we get
\begin{eqnarray}\label{asusu5}
\sum_{n=2}^{\infty}[V_n]_{+} &\leq& \frac{1}{1-\alpha} \sum_{n=1}^{\infty}\delta_n\nonumber \\
&=& \frac{1}{1-\alpha} \sum_{n=2}^{\infty}\delta_n.
\end{eqnarray}
From \eqref{asusu2}, we get
\begin{eqnarray}\label{asusu6}
\sum_{i=1}^n\|x_{i+1}-w_i\|^2 &\leq& \varphi_1-\varphi_n+\alpha \sum_{i=1}^n[V_i]_{+}
+ \sum_{i=2}^n\delta_i\nonumber \\
&\leq& \varphi_1+\alpha C_2+C_1,
\end{eqnarray}
where $C_2=\frac{C_1}{1-\alpha} \geq \frac{1}{1-\alpha}\sum_{i=2}^{\infty}\delta_i \geq \sum_{i=2}^{\infty}[V_i]_{+}$ by
\eqref{asusu5}. Now, since $\varepsilon_1=\varphi_1-\alpha_1\varphi_0=(1-\alpha_1)\varphi_1$, we have
\begin{eqnarray}\label{asusu7}
\varphi_1+\alpha C_2+C_1&=& \varphi_0+\frac{\alpha C_1}{1-\alpha}\nonumber \\
&&+\frac{\alpha(1+\alpha)}{\beta}\Big[\varphi_0+\frac{\varepsilon_1}{1-\alpha}\Big]\nonumber \\
&=& \varphi_0+\frac{\alpha C_1}{1-\alpha}\nonumber \\
&&+\frac{\alpha(1+\alpha)}{\beta}\Big[1+\frac{1-\alpha_0}{1-\alpha}\Big]\varphi_0\nonumber \\
&=& \varphi_0+\frac{\alpha }{1-\alpha}\Big[1+\frac{1-\alpha_0}{1-\alpha} \Big]\nonumber \\
&&+\frac{\alpha(1+\alpha)}{\beta}\Big[1+\frac{1-\alpha_0}{1-\alpha}\Big]\varphi_0\nonumber \\
&=& \Big[1+\Big(\frac{\alpha}{1-\alpha}+\frac{\alpha(1-\alpha)}{\beta} \Big)\Big(1+ \frac{1-\alpha_0}{1-\alpha}\Big) \Big]\varphi_0.
\end{eqnarray}
From \eqref{asusu6} and \eqref{asusu7}, we obtain
\begin{eqnarray}\label{asusu8}
\underset{1 \leq i \leq n}\min \|x_{i+1}-w_i\|^2 \leq \frac{\Big[1+\Big(\frac{\alpha}{1-\alpha}+\frac{\alpha(1-\alpha)}{\beta} \Big)\Big(1+ \frac{1-\alpha_0}{1-\alpha}\Big) \Big]\|x_0-x^*\|^2}{n}.
\end{eqnarray}
\end{proof}

\begin{rem}\label{Rem:DiscussionConvergence}
\noindent
(a) Note that $x_{n+1}=w_n$ implies that $w_n \in C_n$, where $C_n$ is as defined in Algorithm \ref{Alg:AlgL} and hence
$h_n(w_n)\leq 0$. By Lemma \ref{new2}, we get $\frac{\sigma\eta_n}{2}\|w_n-z_n\|^2 \leq h_n(w_n)$. Therefore,
$$
0 \leq \frac{\sigma\eta_n}{2}\|w_n-z_n\|^2 \leq h_n(w_n) \leq 0,
$$
\noindent which implies that $w_n=z_n$. Thus,
the equality $x_{n+1}=w_n$  implies that $x_{n+1}$ is already a solution of VI$(F,C)$ \eqref{bami2}. In this sense, the error estimate given in Theorem \ref{t33} can be viewed as a convergence rate result of the inertial projection-type method \ref{Alg:AlgL}. In particular, \eqref{asusu8} implies that, to obtain an $\epsilon$-optimal solution in the
sense that $\|x_{n+1}-w_n\|^2<\epsilon$, the upper bound of iterations required by inertial projection-type method \ref{Alg:AlgL} is $\frac{\Big[1+\Big(\frac{\alpha}{1-\alpha}+\frac{\alpha(1-\alpha)}{\beta} \Big)\Big(1+ \frac{1-\alpha_0}{1-\alpha}\Big) \Big]\|x_0-x^*\|^2}{\epsilon}$.
We note that with the " $\min_{1\leq i \leq  n}$", a nonasymptotic $O(1/n)$ convergence rate implies that an
$\epsilon$-accuracy solution, in the sense that $\|x_{n+1}-w_n\|^2<\epsilon$, is obtainable within no more than
$O(1/\epsilon)$ iterations. Furthermore, if $\alpha_n = 0$ for all $n$, then the "$\min_{1\leq i \leq  n}$" can be removed by setting $i =n$ in Theorem \ref{t33}.\hfill $\Diamond$
\end{rem}

\section{Numerical Experiments}\label{Sec:Numerics}

In this section, we discuss the numerical behaviour of Algorithm~\ref{Alg:AlgL} using different test examples taken from the literature which are describe below and compare our method with \eqref{pppp3}, \eqref{mm2} and the original Algorithm (when $\alpha_n=0$) of Algorithm \ref{Alg:AlgL}.\\

\begin{exm}
This first example (also considered in \cite{MaingePE1, Malitsky3}) is a classical example for which the usual
gradient method does not converge. It is related to the unconstrained case of VI($F,C$) \eqref{bami2}
where the feasible set is $C:=\mathbb{R}^m$ (for some positive even integer $m$) and
$F:=(a_{ij})_{1\leq i,j\leq m}$ is the square matrix $m \times m$ whose terms are given by

\begin{eqnarray*}
a_{ij}=
\left\{  \begin{array}{llll}
         & -1,~~{\rm if} ~~j=m+1-i~~{\rm and}~~j>i\\
         & 1,~~{\rm if} ~~j=m+1-i~~{\rm and}~~j<i\\
         & 0~~{\rm otherwise}
         \end{array}
         \right.
\end{eqnarray*}
The zero vector $z = (0,\ldots,0)$ is the solution of this test example.
\end{exm}
\noindent
The initial point $x_0$ is the unit vector. We choose $\gamma=0.1$, $\sigma=0.8$ and $\alpha_n=0.6$, $m=500$.
\begin{figure}[!h]
%\begin{flushleft}
\includegraphics[width=13.2cm,height=8.8cm]{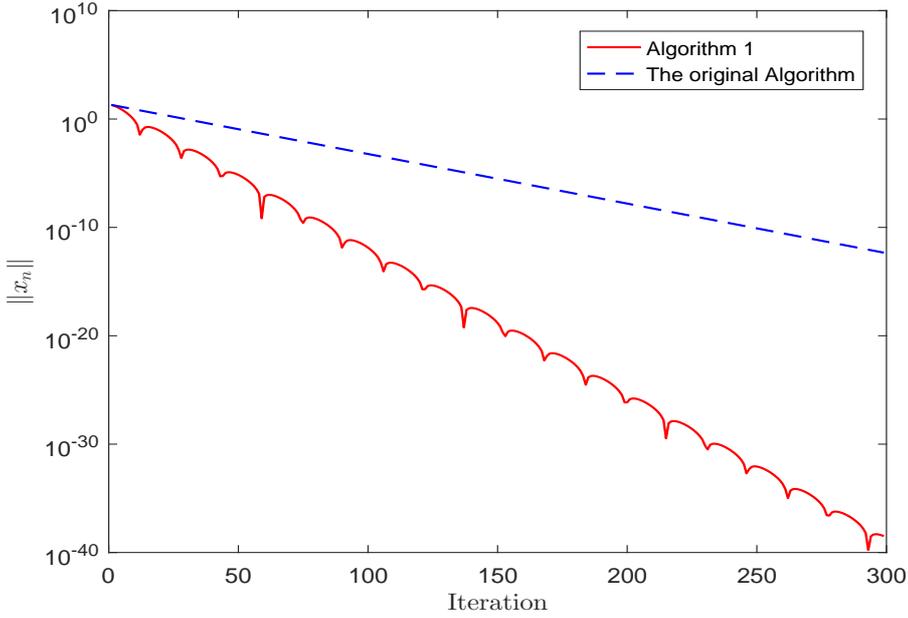}
\caption{Comparison of Algorithm 1 with the original Algorithm.
}
%\end{flushleft}
\end{figure}

\noindent The numerical result is listed in Figure 1, which illustrates that Algorithm 1 highly improves the
original Algorithm.

\begin{exm}\label{her2}
This example is taken from \cite{Harker} and has been
considered by many authors for numerical experiments (see, for example,
\cite{Hieu, Malitsky, Solodov}). The operator $ A $ is defined by
$ A(x) := Mx+q $,
where $ M= BB^T+ S + D $, where $ B, S, D \in
\mathbb R^{m \times m} $ are randomly generated matrices such that $ S $ is
skew-symmetric (hence the operator does not arise from an optimization problem),
$ D $ is a positive definite diagonal matrix (hence the variational inequality
has a unique solution) and $ q = 0 $. The feasible set $ C $ is described by
linear inequality constraints $ B x \leq b $ for some random matrix
$ B \in \mathbb R^{k \times m} $ and a random vector $ b \in \mathbb R^k $
with nonnegative entries. Hence the zero vector is feasible and therefore
the unique solution of the corresponding variational inequality. These projections are computed
by solving a quadratic optimization problem using the MATLAB solver
{\tt quadprog}. Hence, for this class of problems, the evaluation of
$ A $ is relatively inexpensive, whereas projections are costly. We
present the corresponding numerical results (number of iterations and
CPU times in seconds) using four different dimensions $ m $ and two different
numbers of inequality constraints $ k $.
\end{exm}

\noindent
We compare our proposed Algorithm \ref{Alg:AlgL}, original Algorithm, subgradient extragradient method \eqref{pppp3} and the inertial subgradient extragradient method \eqref{mm2} using Example \ref{her2} and the numerical results are listed in Tables 1 -4 and shown in Figures 2-5 below.  We take the initial point $x_0$ to be the unit vector in these algorithms. We use ``OPM" to denote the original Algorithm, ``SPM" to denote the subgradient extragradient method \eqref{pppp3} and ``iSPM" to denote inertial subgradient extragradient method \eqref{mm2}.\\

\noindent We choose the stopping criterion as $\|x^k\|\leq \epsilon = 0.001.$  The size $k =30, 50, 80$ and $m =  20,  50,  80, 100$.   The matrices $B,S,D$ and the vector $b$ are generated randomly. We choose  $\gamma=0.1$, $\sigma=0.8$ and $\alpha_n=0.1$ in Algorithm \eqref{Alg:AlgL}. In \eqref{pppp3}, we choose $\sigma=0.8$, $\rho=0.1$, $\mu=0.2$. In iSPM \eqref{mm2}, $\alpha=0.2, L=\|M\|, \tau=0.5\frac{\frac{1}{2}-2\alpha-\frac{1}{2}\alpha^2}{\frac{1}{2}-\alpha+\frac{1}{2}\alpha^2}$.

\noindent We denote by ``Iter." the number of iterations and ``InIt." the number of total iterations of finding suitable step size in Tables 1 -4 below.

\begin{table}[H]\caption{Comparison of Algorithm 1, original Algorithm and methods \eqref{pppp3} and \eqref{mm2} for $k=50, \alpha_n=0.1.$}\label{table1}
{\tiny
\begin{center}
\begin{tabular}{l|c c c c c c c c c c c c c c }
\hline
&&\multicolumn{1}{c}{Iter. }&&&&&\multicolumn{2}{c}{InIt.}&&&&\multicolumn{2}{c}{CPU in second}&\\
\cline{2-5} \cline{7-9}\cline{11-14}
 $m$ & Alg.1  & OPM &SPM &iSPM && Alg.1  & OPM &SPM && Alg.1  & OPM &SPM &iSPM\\
 \hline
 \hline
  20  &1044&1162&2891&4783&&2376&5022&4283&&0.0781&0.2188&0.6563&0.5781\\
 50   &4829&5912&25544&29639&&13809&30885&123982&&0.5625&0.6094&7.5938&2.1875\\
 80  &19803&22129&19736&61322&&74066&157061&93047&&6.6094&7.0313&15.9063&13.4844\\
 100  &26821&31149&35520&92579&&101925&220297&173670&&12.1406&15.8438&67.4219&34.6250\\
\hline
\hline
\end{tabular}
\end{center}
}
\end{table}

\begin{table}[H]
{\tiny
\caption{Comparison of Algorithm 1, the original Algorithm and methods \eqref{pppp3} and \eqref{mm2} for $k=80, \alpha_n=0.1.$}\label{table2}
\begin{center}
\begin{tabular}{l|c c c c c c c c c c c c c c c }
\hline
&&\multicolumn{1}{c}{Iter. }&&&&\multicolumn{2}{c}{InIt.}&&&&\multicolumn{2}{c}{CPU in second}&\\
\cline{2-5} \cline{7-9}\cline{11-14}
 $m$ & Alg.1  & OPM &SPM &iSPM && Alg.1  & OPM &SPM && Alg.1  & OPM &SPM &iSPM\\
 \hline
 \hline
 20   &1479&1651&5096&5306&&3676&7783&19071&&0.1875&0.2188&6.7869&0.7500\\
 50  &4152&5088&8144&28033&&11955&26594&36342&&1.0625&1.2344&8.2344&6.2813\\
 80  &22711&25864&22281&64588&&85176&182177&105237&&8.2188&9.3438&20.6719&16.1563\\
 100  &26314&30568&37588&88138&&99998&216162&185430&&13.0625&16.1250&118.4375&36.0625\\
\hline
\hline
\end{tabular}
\end{center}
}
\end{table}

\begin{table}[H]
{\tiny
\caption{Comparison of Algorithm 1, the original Algorithm and methods \eqref{pppp3} and \eqref{mm2} for $k=30, \alpha_n=0.6.$}\label{table4}
\begin{center}
\begin{tabular}{l|c c c c c c c c c c c c c  c }
\hline
&&\multicolumn{1}{c}{Iter. }&&&&\multicolumn{2}{c}{InIt.}&&&&\multicolumn{2}{c}{CPU in second}&\\
\cline{2-5} \cline{7-9}\cline{11-14}
 $m$ & Alg.1  & OPM &SPM &iSPM && Alg.1  & OPM &SPM && Alg.1  & OPM &SPM &iSPM\\
 \hline
 \hline
  10  &197&469&509&1327&&341&1370&1460&&1.21884&0.0313&0.2969&0.2500\\
 30   &1548&4745&4347&13697&&4243&17534&16880&&10.8750&0.4063&1.0938&0.8750\\
 50  &1581&4899&8269&26393&&4762&18859&37237&&12.9688&0.5000&2.4688&1.6094\\
 70  &6192&6256&7826&56491&&22381&81445&31406&&62.4844&5.5781&9.3281&9.1813\\
\hline
\hline
\end{tabular}
\end{center}
}
\end{table}

\begin{table}[H]
{\tiny
\caption{Comparison of Algorithm 1, the original Algorithm and methods \eqref{pppp3} and \eqref{mm2} for $k=50, \alpha_n=0.6.$}\label{table4}
\begin{center}
\begin{tabular}{l|c c c c c c c c c c c c c c  }
\hline
&&\multicolumn{1}{c}{Iter. }&&&&\multicolumn{2}{c}{InIt.}&&&&\multicolumn{2}{c}{CPU in second}&\\
\cline{2-5} \cline{7-9}\cline{11-14}
 $m$ & Alg.1  & OPM &SPM &iSPM && Alg.1  & OPM &SPM && Alg.1  & OPM &SPM &iSPM\\
 \hline
 \hline
 10   &147&419&579&1117&&253&1024&1119&&3&0.0313&0.5781&0.3594\\
 30  &1715&5110&6373&13934&&4705&19018&25006&&30.0469&0.4844&1.5469&1.1563\\
 50  &1308&4798&9227&31585&&4062&17873&41084&&41.4375&0.4844&3.5156&2.0469\\
 70  &5673&14944&8205&52124&&20548&74974&33716&&393&4.5469&7.6406&10.4375\\
\hline
\hline
\end{tabular}
\end{center}
}
\end{table}

\begin{figure}[!h]
%\begin{flushleft}
\includegraphics[width=13.2cm,height=8.8cm]{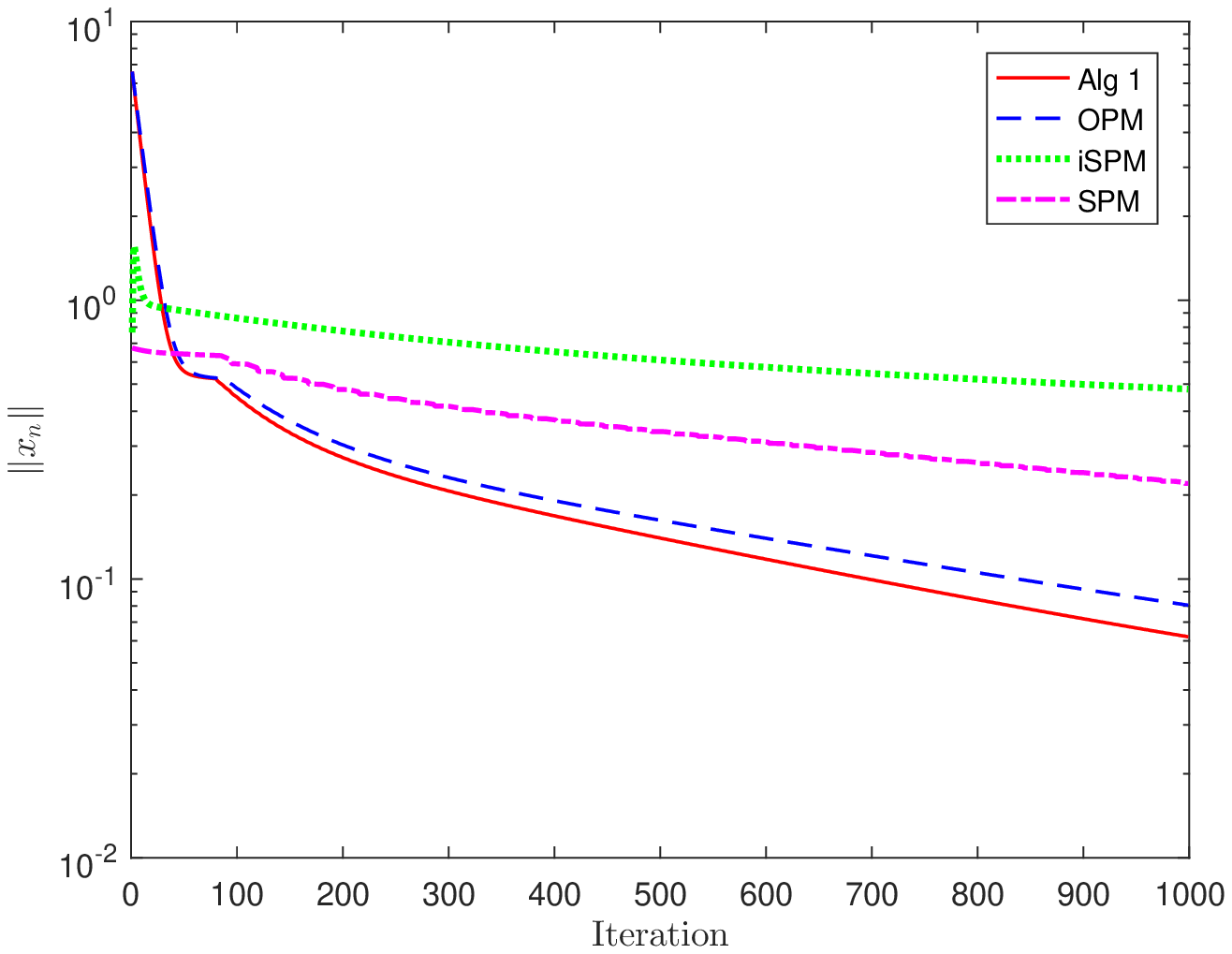}
\caption{Comparison of Algorithm 1, original Algorithm and methods \eqref{pppp3} and \eqref{mm2}. $k=50, \alpha_n=0.1.$
}
%\end{flushleft}
\end{figure}

\begin{figure}[!h]
%\begin{flushleft}
\includegraphics[width=13.2cm,height=8.8cm]{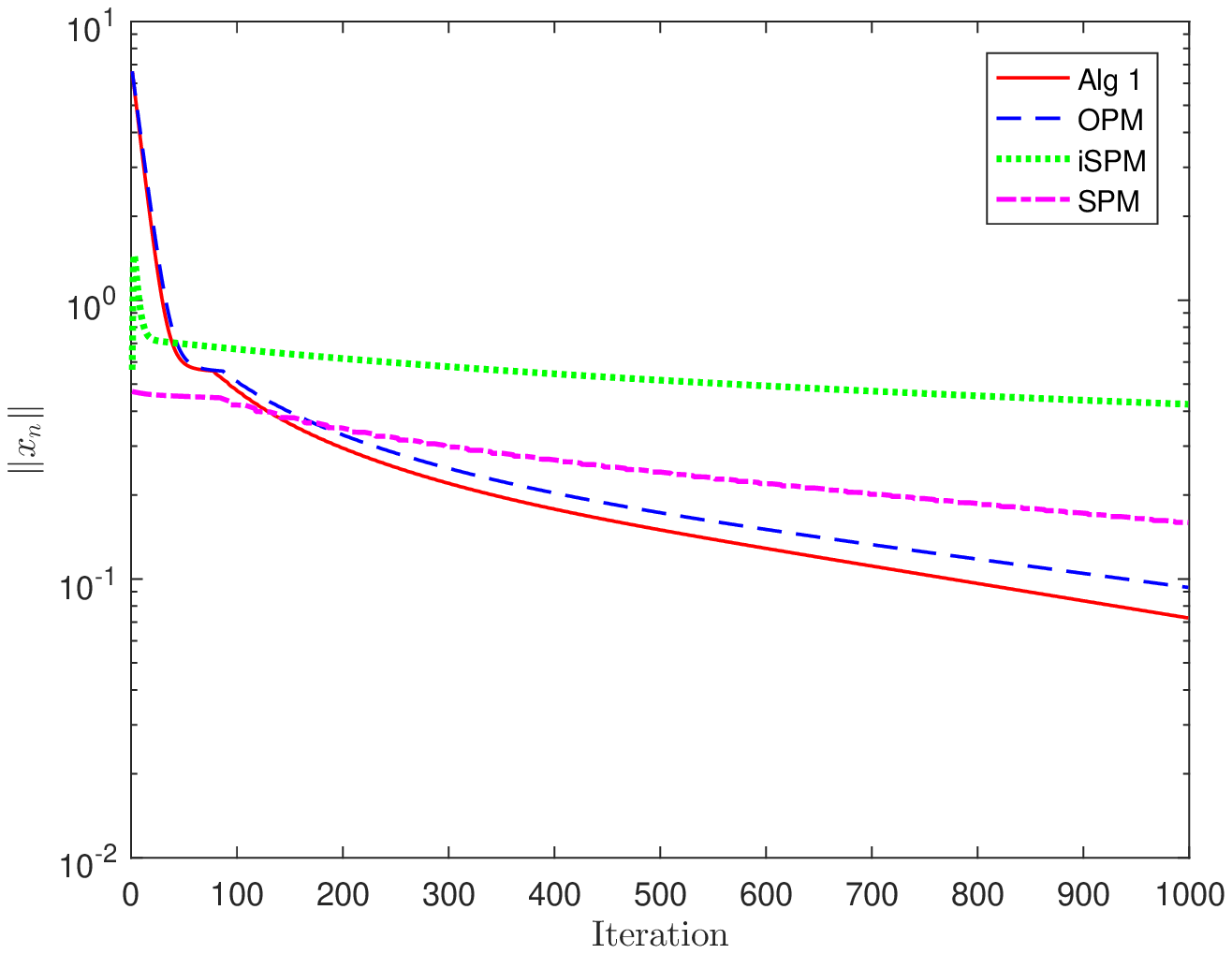}
\caption{Comparison of Algorithm 1, original Algorithm and methods \eqref{pppp3} and \eqref{mm2}. $k=80, \alpha_n=0.1.$
}
%\end{flushleft}
\end{figure}

\begin{figure}[!h]
%\begin{flushleft}
\includegraphics[width=13.2cm,height=8.8cm]{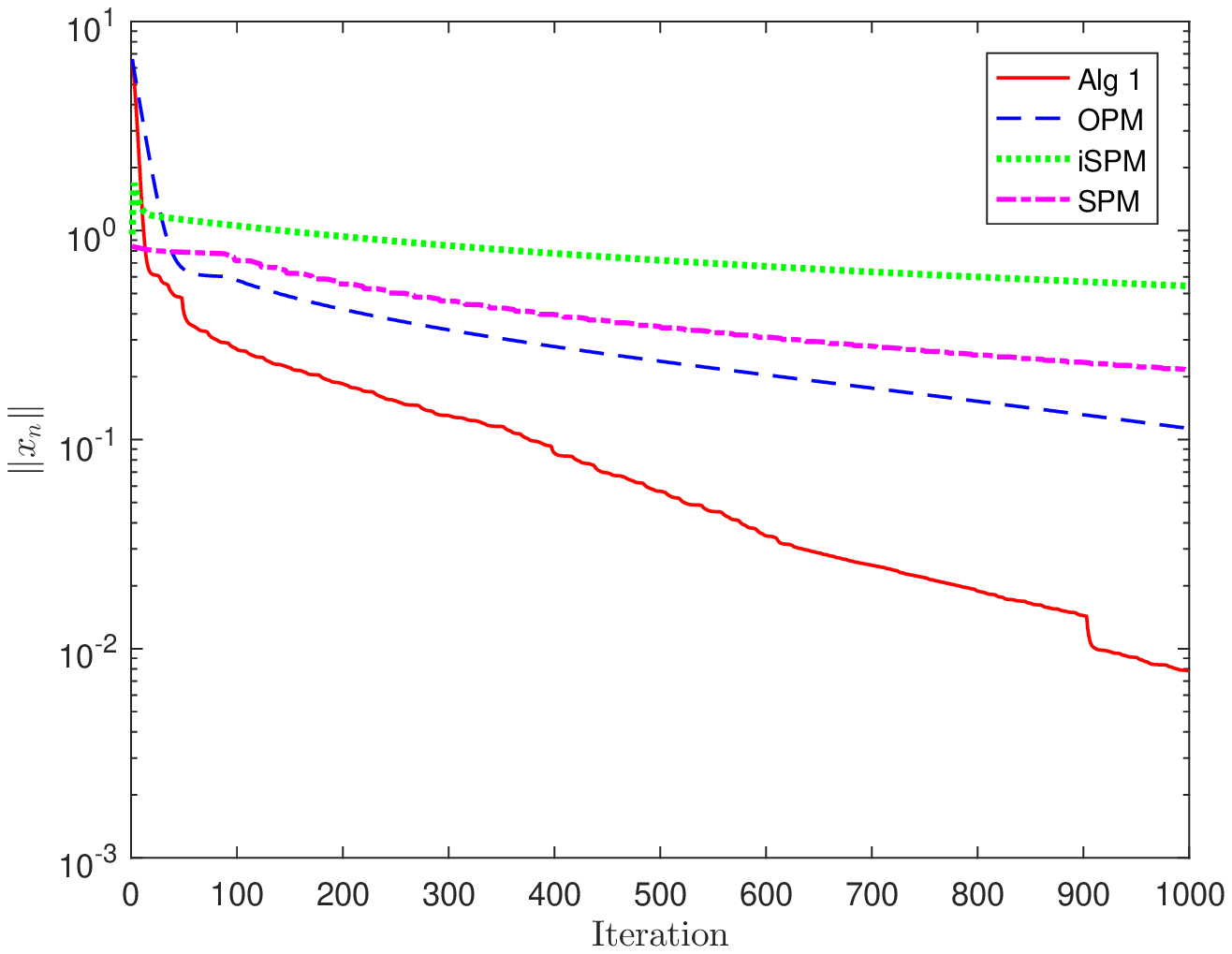}
\caption{Comparison of Algorithm 1, original Algorithm and methods \eqref{pppp3} and \eqref{mm2}. $k=30, \alpha_n=0.6.$
}
%\end{flushleft}
\end{figure}

\begin{figure}[!h]
%\begin{flushleft}
\includegraphics[width=13.2cm,height=8.8cm]{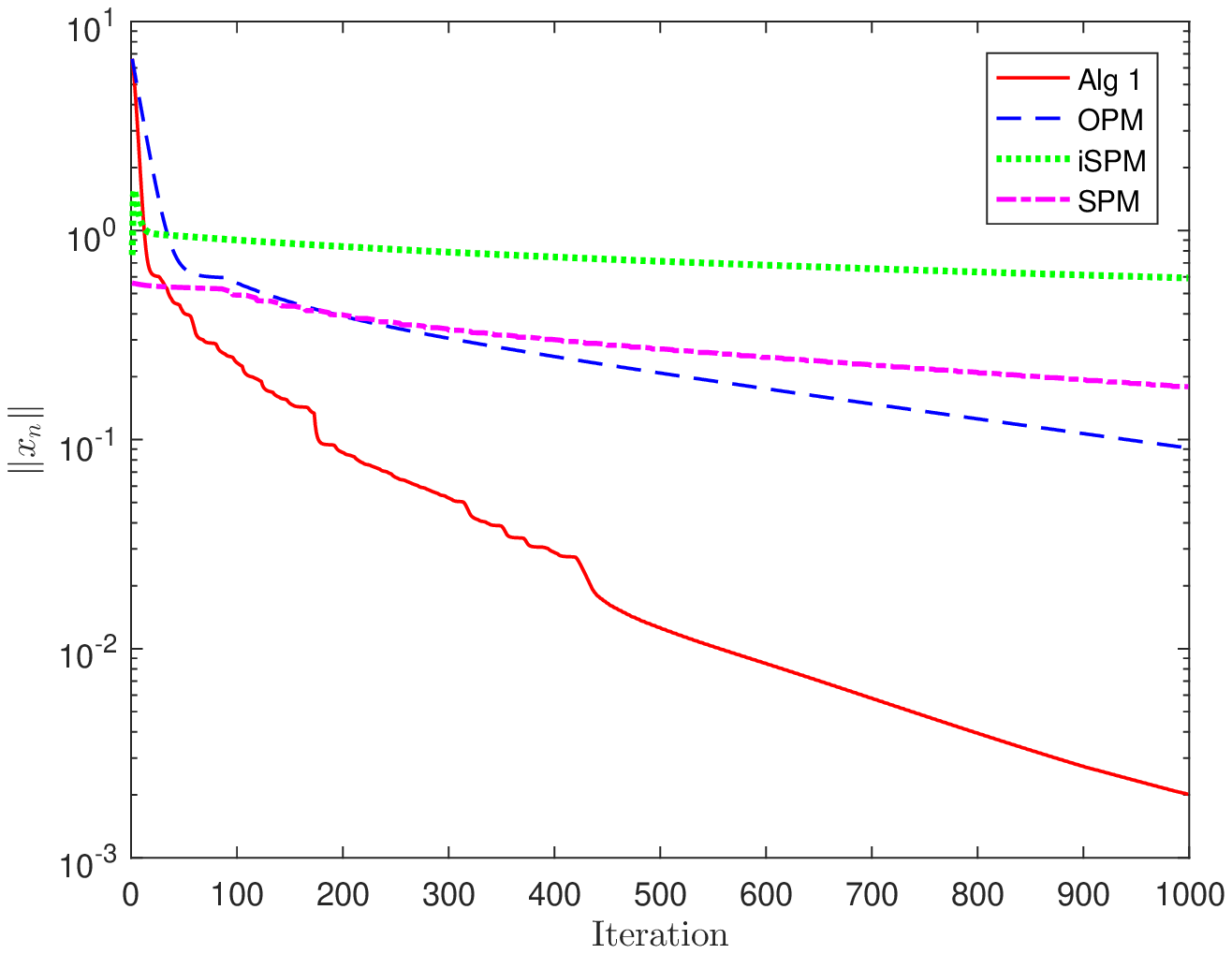}
\caption{Comparison of Algorithm 1, original Algorithm and methods \eqref{pppp3} and \eqref{mm2}. $k=50, \alpha_n=0.6.$
}
%\end{flushleft}
\end{figure}
\noindent
Tables 1 -4 and Figures 2-5 show that Algorithm 1 improves the original Algorithm with respect to ``Iter.", ``InIt." and CPU time. It is also observed from Tables 1 -4 and Figures 2-5 that our proposed Algorithm 1 outperform the subgradient extragradient method \eqref{pppp3} and the inertial subgradient extragradient method \eqref{mm2} with respect to the CPU time and the number of iterations when the feasible set $C$ is nonempty closed affine subset of $H$.

\section{Final Remarks}\label{Sec:Final}

This paper presents a weak convergence result with inertial projection-type method
for monotone variational inequality problems in real Hilbert spaces under very mild assumptions.
This class of method is of inertial nature because at each iteration the projection-type is applied to a point extrapolated at the current iterate in the direction of last movement. Our proposed algorithm framework is not only more simple and intuitive, but also more general than some already proposed inertial projection type methods for solving variational inequality. Based on some pioneering analysis and Algorithm~\ref{Alg:AlgL}, we established certain nonasymptotic $O (1/n) $ convergence rate results. Our preliminary implementation of the algorithms and experimental results have shown that inertial algorithms are generally faster than the corresponding original un-accelerated ones.
In our experiments, the extrapolation step-length $\alpha_n$ was set to be constant. How to select $\alpha_n$ adaptively such that the overall performance is stable and more efficient deserves further investigation. Interesting
topics for future research may include relaxing the conditions on $\{\alpha_n\}$, improving the
convergence results, and proposing modified inertial-type algorithms so that the extrapolation
step-size can be significantly enlarged.\\

\noindent\textbf{Acknowledgements} The project of the first author has received funding from the European Research Council (ERC) under the European Union’s Seventh Framework Program (FP7 - 2007-2013) (Grant agreement No. 616160)

\section{Appendix}
\noindent
In this case, we present a version of Lemma 4.3 for the case when $F$ is pseudo-monotone.

\begin{lem}\label{Lem:help}
Let $F$ be pseudo-monotone, uniformly continuous and sequentially weakly continuous on $H$. Assume that Assumption~\ref{Ass:Parameters} holds.
Furthermore let $\{x_{n_k}\}$ be a subsequence of $ \{x_n\}$
converging weakly to a limit point $ p $. Then $ p \in \text{SOL} $.
\end{lem}

\begin{proof}
By the definition of $z_{n_k}$ together with \eqref{a22}, we have
$$\langle w_{n_k}-F(w_{n_k})-z_{n_k},x-z_{n_k}\rangle\leq 0,\ \forall x\in C,$$
\noindent which implies that
$$\langle w_{n_k}-z_{n_k},x-z_{n_k}\rangle\leq\langle F(w_{n_k}),x-z_{n_k}
  \rangle,\ \forall x \in C.$$
Hence,
\begin{equation}\label{j11}
   \langle w_{n_k}-z_{n_k},x-z_{n_k}\rangle +\langle F(w_{n_k}),z_{n_k}-w_{n_k}\rangle
   \leq \langle F(w_{n_k}),x-w_{n_k}\rangle,\ \forall x\in C.
\end{equation}
Fix $x \in C$ and let $k\rightarrow \infty$ in \eqref{j11}.
Since
$\lim_{k \to \infty} \|w_{n_k}-z_{n_k}\|=0 $, we have
\begin{equation}\label{anbi}
    0\leq \liminf_{k \to \infty} \langle F(w_{n_k}),x-w_{n_k}\rangle
\end{equation}
for all $ x \in C $.
Now we choose a sequence $\{\epsilon_k\}_k$ of positive numbers decreasing and tending to $0$.
For each
$\epsilon_k$, we denote by $N_k$ the smallest positive integer such that
\begin{equation}\label{ine2}
\left\langle  F(w_{n_j}), x- w_{n_j}  \right \rangle + \epsilon_k  \geq 0 \quad \forall j \geq N_k,
\end{equation}
where the existence of $N_k$ follows from \eqref{anbi}.
Since  $\left\lbrace \epsilon_k \right\rbrace $ is decreasing, it is easy to see that
the sequence $ \left\lbrace N_k \right\rbrace  $ is increasing.
Furthermore, for each $k$, $F(w_{N_k}) \not= 0$ and, setting
$$
v_{N_k}=\frac{F(w_{N_k})}{\|F(w_{N_k}\|^2},
$$
we have $ \left\langle  F(w_{N_k}),v_{N_k}  \right \rangle=1$ for each $k$.
Now we can deduce from \eqref{ine2} that for each $k$
$$
\left\langle  F(w_{N_k}),x+\epsilon_k v_{N_k} -w_{N_k}  \right \rangle\geq 0,
$$
and, since $F$ is pseudo-monotone, that
\begin{equation}\label{ine3}
\left\langle  F(x+\epsilon_k v_{N_k}),x+\epsilon_k v_{N_k} -w_{N_k}  \right \rangle\geq 0.
\end{equation}
On the other hand, we have that $\left\lbrace x_{n_k} \right\rbrace $ converges weakly to $p$ when $k \to \infty$.
Since $F$ is sequentially weakly continuous on $C$, $\left\lbrace F(w_{n_k}) \right\rbrace $ converges weakly to $F(p)$.
We can suppose that $F(p) \not = 0$ (otherwise, $p$ is a solution).
Since the norm mapping is sequentially weakly lower semicontinuous, we have
$$
0<\|F(p) \| \leq \lim \inf_{k \to \infty} \|F(w_{n_k})\|.
$$
Since $\left\lbrace w_{N_k} \right\rbrace \subset \left\lbrace w_{n_k} \right\rbrace$ and $\epsilon_k \to 0$ as $k \to \infty$, we obtain
\begin{eqnarray*}
0 &\leq& \lim \sup_{k \to \infty} \|\epsilon_k v_{N_k}\| = \lim \sup_{k \to \infty} \Big(\frac{\epsilon_k}{\|F(w_{n_k})\|}\Big)\\
&\leq& \frac{\lim \sup_{k \to \infty}\epsilon_k}{\lim \inf_{k \to \infty}\|F(w_{n_k})\|}  \leq \frac{0}{\|F(p)\|}=0,
\end{eqnarray*}
which implies that $\lim_{k \to \infty} \|\epsilon_k v_{N_k}\|=0$.
Hence, taking the limit as $k \to \infty$ in \eqref{ine3}, we obtain
$$
\left\langle  F(x),x-p  \right \rangle \geq 0.
$$
\noindent
Now, using Lemma~2.2 of \cite{Mashreghi}, we have that $p\in\text{SOL}$.
\end{proof}

\end{document}